\documentclass[11pt]{article}
\usepackage[utf8]{inputenc}
\usepackage{graphicx}         \usepackage{amsmath}            \usepackage{amsthm}
\usepackage{amssymb}          \usepackage{psfrag}
\usepackage{paralist}         \usepackage{multirow}           \usepackage{bigdelim}
\usepackage{mathrsfs}         \usepackage{longtable}          \usepackage{hhline}
\usepackage{color}            

\newtheorem{theorem}{Theorem}[section]
\newtheorem{proposition}{Proposition}[section]

\newtheorem{rem}{Remark}[section]
\newtheoremstyle{remark}            
 {0.5\topsep}                       
 {0.5\topsep}                       
 {\upshape\small}                   
 {}                                 
 {\itshape}                         
 {}                                 
 {.5em}                             
 {}                                 
 {}                                 
\theoremstyle{remark}               

\DeclareMathOperator{\matchOp}    {M}

\DeclareMathOperator{\perfMatchOp}{PM}
\DeclareMathOperator{\inOp}       {in}
\DeclareMathOperator{\outOp}      {out}
\DeclareMathOperator{\sucOp}      {succ}

\DeclareMathOperator{\conv}       {conv}

\DeclareMathOperator{\neighOp}    {N}
\DeclareMathOperator{\critOp}     {crit}

\newcommand{\circOp}            {\circlearrowright}
\newcommand {\crit}[1]          {\critOp({#1})}
\newcommand{\family}[1]         {\mathscr{#1}}
\newcommand {\graphstar}[2]     {\delta_{#2}(#1)}
\newcommand{\instar}[2]         {\delta^{\inOp}_{#2}(#1)}

\newcommand{\mat}[1]            {\mathbf{#1}}

\newcommand{\N}                 {\mathbb{N}}

\newcommand{\nb}[1]             {\neighOp({#1})}
\newcommand{\nonNb}[1]          {\overline{\neighOp}({#1})}

\newcommand{\orb}[1]            {\polytope{O}_{#1}}
\newcommand{\orbisack}[1]       {\polytope{O}_{#1}}
\newcommand{\orbverts}[1]       {\set{X}_{#1}}
\newcommand{\outstar}[2]        {\delta^{\outOp}_{#2}(#1)}

\newcommand{\polyMatch}[1]      {\polytope{P}_{\matchOp}({#1})}
\newcommand{\polyPerfMatch}[1]  {\polytope{P}_{\perfMatchOp}({#1})}
\newcommand{\polytope}[1]       {\mathrm{#1}}

\newcommand{\R}                 {\mathbb{R}}
\newcommand{\RNonNeg}           {\mathbb{R}_+}
\newcommand{\set}[1]            {\mathcal{#1}}
\newcommand{\setdef}[2]         {\{#1\;|\;#2\}}
\newcommand{\setDef}[2]         {\setdef{#1}{#2}}
\newcommand{\ssmin}             {\smallsetminus}

\newcommand{\successors}[2]     {\sucOp_{#2}({#1})}
\newcommand{\unitvec}[1]        {\mathbb{e}^{#1}}
\newcommand{\zerovec}[1]        {\mathbb{O}_{#1}}
\newcommand{\zeroVec}           {\zerovec{}}
\newcommand{\oneVec}            {\mathbb{1}}
\newcommand{\st}                {^{\star}}
\newcommand{\scalProd}[2]       {\langle{#1}\,,\,{#2}\rangle}

\newcommand{\transpose}[1]      {{#1}^{T}}

\renewcommand{\vec}[1]          {\boldsymbol{#1}}

\def\clap#1{\hbox to 0pt {\hss#1\hss}}

\title{Finding Descriptions of Polytopes via Extended Formulations and Liftings}
\author{Volker Kaibel and Andreas Loos}

\date{\today}

\begin{document}

\maketitle

\begin{abstract}
 We describe a technique to obtain linear descriptions for polytopes from extended formulations. The simple idea is to first define a suitable lifting function and then to find linear constraints that are valid for the polytope and guarantee lifted points to be contained in the extension. We explain the technique at an example from the literature (matching polytopes), obtain new simple proofs of results on path-set polytopes and small-cliques polytopes, and finally exploit the technique in order to derive linear descriptions of orbisacks, which are special Knapsack polytopes arising in the context of symmetry breaking in integer programming problems. 
\end{abstract}

 \section{Introduction}

Describing polytopes that encode combinatorial problems by means of  systems of linear equations and inequalities is a crucial topic in Combinatorial Optimization, because this approach, known as \emph{Polyhedral Combinatorics}, makes combinatorial optimization problems accessible to linear programming techniques. While the Weyl-Minkowski Theorem guarantees that for every polytope (i.e., the convex hull of a finite set of points)  such a description \emph{exists}, it can be quite hard to actually \emph{find} some. Sometimes, it is much easier to derive a linear description of some higher dimensional polyhedron that can be projected to the polytope in question by some linear (or affine) map. Such a  description, known as \emph{extended formulation} (see, e.g., \cite{VW10,CCZ10,Kai11}), can be used instead of the original polytope. But sometimes, extended formulations can also be exploited in order to find descriptions in the original spaces.

The classical method to do this is by finding a generating set of the \emph{projection cone}. In order to explain this, let us look at the (in fact, not really restrictive) case of a polytope $\polytope{P}\subseteq\R^n$ that is the orthogonal projection $\polytope{P}=\setDef{x\in\R^n}{(x,y)\in\polytope{Q}\text{ for some }y\in\R^q}$ of a polyhedron~$Q\subseteq\R^n\times\R^q$. For a description $\polytope{Q}=\setDef{(x,y)\in\R^n\times\R^q}{Ax+By\le b}$ of~$\polytope{Q}$ by linear inequalities (with $A\in\R^{m\times n}$, $B\in\R^{m\times q}$, and $b\in\R^m$), the polyhedral cone
\begin{equation*}
	\polytope{C}=\setDef{\lambda\in\R_+^m}{\transpose{\lambda}B=\zeroVec}
\end{equation*} 
is called the \emph{projection cone}. If~$\Lambda\subseteq \polytope{C}$ is a finite set of generators of~$\polytope{C}$ (i.e., every $\lambda\in \polytope{C}$ can be written as a linear combination of vectors from~$\Lambda$ with nonnegative coefficients), then
\begin{equation*}
	(\transpose{\lambda}A)x \le \transpose{\lambda}b\qquad\forall\lambda\in\Lambda
\end{equation*}
is a system of inequalities describing~$\polytope{P}$ (see, e.g., \cite{CCZ10}). Thus, in order to derive a linear description of some polytope from an extended formulation, it is enough to find a finite set of generators of the associated projection cone, e.g., by determining its extreme rays. In some cases, this method has been applied very successfully. It is worth to note that, while dealing with a projection of a polytope given by linear inequalities is  non-trivial (in general, computing generators of the projection cone is a difficult task), the image of a polytope that is  given as the convex hull  of some set clearly is the convex hull of the projection of that set.

In this paper, we describe an alternative method for deriving linear descriptions of polytopes from extended formulations that we call the \emph{lifting method}. Actually, the method is not  new. It is, e.g.,  a generalization of the method used in~\cite{Schrijver04} in order to deduce  descriptions of   matching polytopes from the descriptions of  perfect matching polytopes (see the proof of Cor.~25.1a in~\cite{Schrijver04}). Our contribution here is meant to first of all draw attention to the method itself (Section~\ref{sec: lifting example}), to demonstrate its capabilities by providing alternative derivations of well-known linear descriptions (of  path-set polytopes in Section~\ref{sec: path set polytopes} and of  small-cliques polytope in Section~\ref{sec: small clique polytopes}), and finally to use the method in order to derive linear descriptions of a special class of Knapsack polytopes, the \emph{orbisacks}, which arise in the context of symmetry breaking in integer programming models. 

Most of the material of this paper can also be found in the PhD-disser\-ta\-tion~\cite{Loos11}.

 \section{The Lifting-Method}
\label{sec: lifting example}

As an introductory example, we deal with we use the derivation of the linear description of the \emph{matching polytope} from the description of the \emph{perfect matching polytope}  (i.e., the convex hulls of the characteristic vectors in $\R^{\set{E}}$ of all respectively of all perfect matchings in a graph~$G=(\set{V},\set{E})$) as given in the proof of Cor.~25.1a in~\cite{Schrijver04}.
We will denote these polytopes by 
\[
 \polyMatch{G}=\conv\setdef{\vec{x}[\set{M}]\in\{0,1\}^{\set{E}}}{\set{M} \text{ matching in }G}
\]
and 
\[
 \polyPerfMatch{G}=\conv\setdef{\vec{x}[\set{M}]\in\{0,1\}^{\set{E}}}{\set{M} \text{ perfect matching in }G}
\]
(where $\vec{x}[\cdot]$ denotes the \emph{characteristic vector} of the set in the brackets, i.e., the 0/1-vector having one-entries exactly at positions indexed by that set).

We fix by $G_1=(\set{V}_1,\set{E}_1)$ and $G_2=(\set{V}_2,\set{E}_2)$ two disjoint copies of~$G$. For a vertex $v\in \set{V}$ or a set $\set{W}\subseteq \set{V}$ of vertices of~$G$ we denote by~$v_1$, $v_2$, $\set{W}_1$, and $\set{W}_2$ the respective copies in~$G_1$ and $G_2$. The graph  $\tilde{G}=(\tilde{\set{V}},\tilde{\set{E}})$ arises from~$G_1$ and~$G_2$ by connecting~$v_1$ to~$v_2$ for each~$v\in \set{V}$. It is easy to see that~$\polyMatch{G}=\polyMatch{G_1}$ is the orthogonal projection of~$\polyPerfMatch{\tilde{G}}$ to the $\set{E}_1$-coordinates.  

In order to describe the method in general, let~$\polytope{Q}\subseteq\R^d$ be a polyhedron whose image under the projection $\sigma:\R^d\rightarrow\R^n$
 is the polytope~$\polytope{P}\subseteq\R^n$. In our example, we have $\polytope{P}=\polyMatch{G}$ and $\polytope{Q}=\polyPerfMatch{\tilde{G}}$.
For the applicability of the method it is crucial to find a suitably described \emph{lifting} $\lambda:\set{R}\rightarrow\R^d$ on a set $\set{R}\subseteq\R^n$ containing~$\polytope{P}$ with
\begin{equation*}
	\sigma(\lambda(\vec{x}))=\vec{x}\quad\text{for all }\vec{x}\in \set{R}\,.
\end{equation*}
 In the matching example, we choose 
\begin{equation*}
	\set{R} =\setDef{\vec{x}\in\R^{\set{E}}}{\vec{x}\ge\zeroVec,x(\graphstar{v}{}))\le 1\text{ for all }v\in \set{V}}
\end{equation*}
(where, as usual, we denote by $\graphstar{v}{}$ the set of edges incident to~$v$, and, for some vector $\vec{x}$, by $x(\cdot)$  the sum of all components of~$\vec{x}$ indexed by elements from the set in the brackets).
Actually, the method can only work if the lifting satisfies $\lambda(\vec{x})\in \polytope{Q}$ for all $\vec{x}\in \polytope{P}$, i.e., the restriction of the lifting  to $\polytope{P}$ is a \emph{section} of the extension. However, this property needs not to be established explicitly, but it rather follows in hindsight if the method works out. At this point, the requirement is only used to guide the search for a promising lifting.  For instance, looking at the vertices of $\polyMatch{G}$ one may find $\lambda:\set{R}\rightarrow\R^{\tilde{\set{E}}}$ with 
\begin{align*}
\lambda(\vec{x})_{\set{E}_1} &= \vec{x}\\
\lambda(\vec{x})_{\set{E}_2} &= \vec{x}\\
\lambda(\vec{x})_{\{v_1,v_2\}} &= 1-x(\graphstar{v}{})
\end{align*}
for all $\vec{x}\in \set{R}$ and $v\in \set{V}$  to be a natural choice for the matching example.

Suppose we have a linear description of $\set{R}$ at hand. In order to find a linear description of~$\polytope{P}$ it then suffices to exhibit a system $\mat{A}\vec{x}\le \vec{b}$ of inequalities valid for~$\polytope{P}$ that is \emph{section enforcing} (with respect to~$Q$, $\sigma$, and~$\lambda$),
 i.e., $\lambda(\vec{x})\in \polytope{Q}$ is satisfied for all $\vec{x}\in \set{R}$ with $\mat{A}\vec{x}\le \vec{b}$. Indeed, in this case we clearly have $\polytope{P}\subseteq\setDef{\vec{x}\in \set{R}}{\mat{A}\vec{x}\le \vec{b}}$ (as $\mat{A}\vec{x}\le \vec{b}$ is valid for~$\polytope{P}$), and the reverse inclusion follows from $\vec{x}=\sigma(s(\vec{x}))$ for all $\vec{x}\in \set{R}$, $s(\vec{x})\in \polytope{Q}$ for all $\vec{x}\in \set{R}$ with $\mat{A}\vec{x}\le \vec{b}$, and $\polytope{P}=\sigma(\polytope{Q})$. 

In case of the matching example, we can find a section enforcing system of valid inequalities for $\polytope{P}=\polyMatch{G}$ as follows, exploiting the fact that $\polytope{Q}=\polyPerfMatch{\tilde{G}}$ equals the set of all $\tilde{\vec{x}}\in\RNonNeg^{\tilde{E}}$ that satisfy
\begin{equation}\label{eq:PMNode}
	\tilde{x}(\graphstar{\tilde{v}}{})=1\quad\text{for all }\tilde{v}\in\tilde{\set{V}}
\end{equation}
and
\begin{equation}\label{eq:PMOddSet}
	\tilde{x}(\graphstar{\tilde{\set{W}}}{})\ge 1\quad\text{for all }\tilde{\set{W}}\subseteq\tilde{\set{V}}, |\tilde{\set{W}}|\text{ odd}\,.
\end{equation}
For $\vec{x}\in \set{R}$ and $\tilde{\vec{x}}=\lambda(\vec{x})$ we have $\tilde{\vec{x}}\ge\zeroVec$ as well as~\eqref{eq:PMNode} by definition. Thus it remains to identify linear inequalities that are valid for~$\polyMatch{G}$ and imply~\eqref{eq:PMOddSet}. In order to accomplish that task, let $\tilde{\set{W}}\subseteq\tilde{\set{V}}$ be any set of odd cardinality, and let
\begin{eqnarray*}
	\set{A} &=& \setDef{v\in \set{V}}{v_1    \in\tilde{\set{W}},v_2\not\in\tilde{\set{W}}}\,,\\
	\set{B} &=& \setDef{v\in \set{V}}{v_1    \in\tilde{\set{W}},v_2    \in\tilde{\set{W}}}\,,\\
	\set{C} &=& \setDef{v\in \set{V}}{v_1\not\in\tilde{\set{W}},v_2    \in\tilde{\set{W}}}\text{ and}\\
\end{eqnarray*}
$\set{D}=\set{V}\setminus(\set{A}\cup \set{B}\cup \set{C})$. 
We have
\begin{eqnarray}\label{eq:success}
	\tilde{x}(\graphstar{\tilde{\set{W}}}{})
	    &\ge& 
		\tilde{x}(\set{A}_1:\set{A}_2) + \tilde{x}(\set{B}_2:\set{A}_1) + \tilde{x}(\set{A}_1:\set{C}_1) + \tilde{x}(\set{A}_1:\set{D}_1)\nonumber \\
		&=& \sum_{v\in \set{A}}\big(1-x(\graphstar{v}{}\big) + x(\graphstar{\set{A}}{})\\
		&=& |\set{A}|-2x(\set{E}[\set{A}])\,,\nonumber
\end{eqnarray}
(where $\graphstar{\cdot}{}$, $(\cdot:\cdot)$, and $\set{E}[\cdot]$ are the sets of edges with exactly one end-node in the set in the brackets, one end-node in the first and one in the second set, and both end-nodes in the set, respectively)
and, similarly, $\tilde{x}(\graphstar{\tilde{\set{W}}}{}) \ge |\set{C}|-2x(\set{E}[\set{C}])$. Hence, \eqref{eq:PMOddSet} holds as soon as  
\begin{equation*}
	x(\set{E}[\set{A}])\le \frac{|\set{A}|-1}{2}
	\quad\text{or}\quad
	x(\set{E}[\set{C}])\le \frac{|\set{C}|-1}{2}
\end{equation*}
is satisfied. Fortunately, since $|\tilde{\set{W}}|$ is odd, $|\set{A}|$ or~$|\set{C}|$ must be odd, therefore the system
\begin{equation}\label{eq:blossom}
	x(\set{E}[\set{S}])\le\frac{|\set{S}|-1}{2}
	\quad\text{for all }\set{S}\subseteq \set{V}, |\set{V}|\text{ odd}
\end{equation} 
is section enforcing and valid for $\polyMatch{G}$, which establishes
\begin{equation*}
	\polyMatch{G}=\setDef{\vec{x}\in \set{R}}{\vec{x} \text{ satisfies~\eqref{eq:blossom}}}\,.
\end{equation*}

As one may see from~\eqref{eq:success}, the success of the method crucially depends on the availability of a lifting that is described in a way exploitable for establishing membership in the extension polyhedron. In the matching example, the lifting~$\lambda$ was an affine map whose defining formulas could be plugged immediately into the linear description of the extension polyhedron $\polyPerfMatch{\tilde{G}}$. In fact, it is not necessary that the lifting is of linear type. In the application to orbisacks in Section~\ref{sec:orbi} the liftings will indeed be only piecewise affine, and in the two applications worked out in the next two sections the liftings will even not be defined by explicit formulas at all.

As for the matching example, in many cases a fruitful way to come up with a useful lifting seems to be to try to find a natural way to lift the vertices of~$\polytope{P}$ into~$\polytope{Q}$ first, and then to try to define a (usable description of a) lifting map on the whole set~$\set{R}$ extending that lifting of the vertices. Often  a lifting of the vertices is rather obvious from combinatorial considerations. It may also be known already from establishing $\polytope{P}\subseteq \sigma(\polytope{Q})$.

  \section{Path Set Polytopes}
\label{sec: path set polytopes}

\def\Qstar{\polytope{Q}_{\circOp}(\tilde{D},\vec{u}^{\star},\vec{\ell}^{\star})}
\def\Qx{\polytope{Q}_{\circOp}(\tilde{D},\vec{u}^{\vec{x}},\vec{\ell}^{\vec{x}})}

In our second example, we derive a linear description of the $s$-$t$-path set polytope 
\[
 \polytope{P}^{s,t}(D)=\conv\setdef{\vec{x}[\set{P}]\in\{0,1\}^{\set{V}}}{\set{P} \text{ is the node set of some }s\text{-}t\text{-path in }D}
\]
of an \emph{acyclic} digraph $D=(\set{V},\set{A})$ with two nodes $s\ne t$, where, for technical reasons, we assume that~$s$ is a source of~$D$.
Our derivation only reproves a result of Vande Vate's \cite{VandeVate89}, whose proof works via analyzing the projection cone (phrased in terms of Benders' cuts) of basically the same extended formulation as we are going to exploit.

Just like in the matching example, we use an extended formulation based on a directed graph $\widetilde{D} = (\widetilde{\set{V}},\widetilde{\set{A}})$ with a node set
 $\widetilde{\set{V}}$
that contains two clone nodes -- now denoted $v^{\inOp}$ and $v^{\outOp}$ -- of each $v\in \set{V}$. 
The arc set is defined as
\[
 \widetilde{\set{A}} = \setdef{(v^{\outOp},w^{\inOp})}{(v,w)\in\set{A}} \;\cup\; \setdef{(v^{\inOp},v^{\outOp})}{v\in\set{V}}  \;\cup\; (t^{\outOp},s^{\inOp}).
\]
Arcs in $\setdef{(v^{\outOp},w^{\inOp})}{(v,w\in\set{A})}$ will be referred to as \emph{real arcs}; see Figure~\ref{fig: faithful sectioning path polytope extension} for an example.
\begin{figure}
 \centering
 \psfrag{sout}{$\scriptscriptstyle s^{\outOp}$} \psfrag{sin}{$\scriptscriptstyle s^{\inOp}$}
 \psfrag{tout}{$\scriptscriptstyle t^{\outOp}$} \psfrag{tin}{$\scriptscriptstyle t^{\inOp}$}
 \psfrag{vout}{$\scriptscriptstyle v^{\outOp}$} \psfrag{vin}{$\scriptscriptstyle v^{\inOp}$}
 \psfrag{s}{$s$}             \psfrag{t}{$t$}
 \psfrag{v}{$v$}
 \includegraphics[width=.7\textwidth]{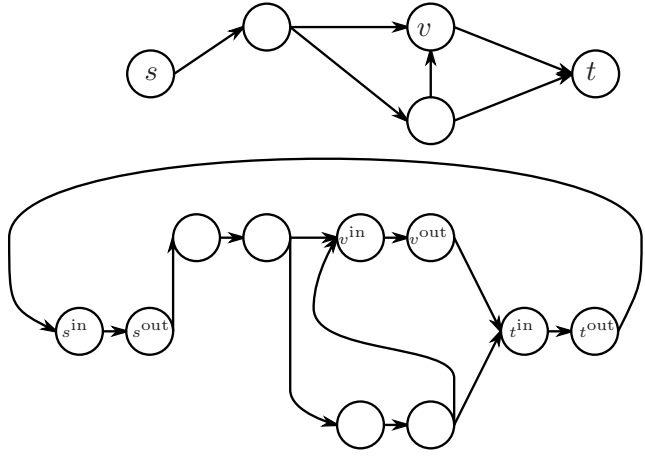}
 \caption{Example digraph $\widetilde{D}$ (bottom) obtained from acyclic digraph $D$ (top).\label{fig: faithful sectioning path polytope extension}}
\end{figure}
The extension we are going to use is  the  polytope 
\begin{multline*}
	\polytope{Q}_{\circOp}(\widetilde{D},\vec{\ell}^{\star},\vec{u}^{\star})
	=\\
	\setDef{\vec{y}\in\R^{\widetilde{\set{A}}}}{y(\outstar{\tilde{v}}{\widetilde{D}})=y(\instar{\tilde{v}}{\widetilde{D}})\text{ for all }\tilde{v}\in\widetilde{\set{V}}, \vec{\ell}^{\star}\le\vec{y}\le\vec{u}^{\star}}
\end{multline*} 
 of circulations in the digraph $\widetilde{D}$ obeying the following capacities:
  \[
   \begin{array}{l@{\;}ll}
    \ell_{(v^{\outOp},w^{\inOp})}^{\star}     &= -\infty & \rdelim\}{2}{0em}[\;for all $(v,w)\in\set{A}$] \\
    u_{(v^{\outOp},w^{\inOp})}^{\star}        &= +\infty & \\
    \ell_{(v^{\inOp},v^{\outOp})}^{\star}     &= 0       & \rdelim\}{2}{0em}[\;for all $v \in\set{V}$] \\
    u_{(v^{\inOp},v^{\outOp})}^{\star}        &= +\infty & \\
    \ell_{(t^{\outOp},s^{\inOp})}^{\star}     &= 1       & \\
    u_{(t^{\outOp},s^{\inOp})}^{\star}        &= 1       &
   \end{array}
  \]
   The vertices of~$\polytope{Q}_{\circOp}(\widetilde{D},\vec{\ell}^{\star},\vec{u}^{\star})$ correspond to the directed cycles in~$\widetilde{D}$, all of which contain~$(t^{\outOp},s^{\inOp})$ (as~$D$ is acyclic). Thus we have $\sigma(\polytope{Q}_{\circOp}(\widetilde{D},\vec{\ell}^{\star},\vec{u}^{\star}))=\polytope{P}^{s,t}(D)$
  with~$\sigma:\R^{\widetilde{\set{A}}}\to\R^{\set{V}}$  defined via $\sigma(\vec{y})_v = y_{(v^{\inOp},v^{\outOp})}$. 
 
In order to define a suitable lifting~$\lambda$,
 observe that for  
\begin{equation*}
	x\ \in\ \set{R}=\setDef{x\in\R_+^{\set{V}}}{x_s=x_t=1}\supseteq\polytope{P}^{s,t}(D)
\end{equation*}
and $y\in\polytope{Q}_{\circOp}(\widetilde{D},\vec{\ell}^{\star},\vec{u}^{\star})$ we have $\sigma(y)=x$ if and only if $y\in\polytope{Q}_{\circOp}(\widetilde{D},\vec{\ell}^{\vec{x}},\vec{u}^{\vec{x}})$  holds with~$\vec{\ell}^{\vec{x}}$ and~$\vec{u}^{\vec{x}}$ being equal to~$\vec{\ell}^{\star}$ and~$\vec{u}^{\star}$ in all components except for $\ell_{(v^{\inOp},v^{\outOp})}^{\vec{x}}=u_{(v^{\inOp},v^{\outOp})}^{\vec{x}}=x_v$ for all $v\in\set{V}$.
Consequently, for all $\vec{x}\in\set{R}$ we define~$\lambda(\vec{x})$ to be an arbitrary point in $\polytope{Q}_{\circOp}(\widetilde{D},\vec{\ell}^{\vec{x}},\vec{u}^{\vec{x}})$ if this set of circulations is nonempty, and (just for formal reasons) to be an arbitrary point in~$\R^{\widetilde{\set{A}}}$ with $\lambda(\vec{x})_{(v^{\inOp},v^{\outOp})}=x_v$ for all $v\in\set{V}$ otherwise.

Clearly, we have $\sigma(\lambda(\vec{x}))=x$ for all $\vec{x}\in\set{R}$, and thus,
it remains to find a system of inequalities that is valid for~$\polytope{P}^{s,t}(D)$ and section enforcing, where the latter condition in this case just means that $\polytope{Q}_{\circOp}(\widetilde{D},\vec{\ell}^{\vec{x}},\vec{u}^{\vec{x}})\ne\varnothing$ holds for  every $\vec{x}\in\set{R}$ satisfying the system. The crucial characterization of the existence of circulations that we exploit here is \emph{Hoffman's Circulation Theorem}.

\begin{theorem}[Hoffman's Circulation Theorem, \cite{Hoffman60}]\label{theo: hoffman's circulation theorem}
In a digraph with lower and upper arc capacities vectors~$\ell$ and~$u$ (with components from $\R\cup\{-\infty,+\infty\}$)  a circulation exists 
 if and only if
\[
 \ell\big(\instar{\set{W}}{}\big) \leq  u\big(\outstar{\set{W}}{}\big)
\]
holds for all node subsets $\set{W}$ (where $\instar{\cdot}{}$ and $\outstar{\cdot}{}$ are the sets of all arcs pointing into and out of, respectively, the set in brackets). 
\end{theorem}

Thus, in order to guarantee $\polytope{Q}_{\circOp}(\widetilde{D},\vec{\ell}^{\vec{x}},\vec{u}^{\vec{x}})\ne\varnothing$ for some $x\in\set{R}$, we have to ensure
\begin{equation}\label{eq:PSgoal}
	\ell^{\vec{x}}\big(\instar{\widetilde{\set{S}}}{\widetilde{D}}\big) \leq u^{\vec{x}}\big(\outstar{\widetilde{\set{S}}}{\widetilde{D}}\big)
\end{equation}
for all $\widetilde{\set{S}}\subseteq \widetilde{\set{V}}$. Clearly, we only have to care about subsets $\widetilde{\set{S}}\subseteq \widetilde{\set{V}}$ such that
\begin{equation}\label{eq:specialProp}
	\instar{\widetilde{\set{S}}}{\widetilde{D}}
	\cup
	\outstar{\widetilde{\set{S}}}{\widetilde{D}}
	\quad
	\text{does not  contain any real arc.}
\end{equation}
Let  $\widetilde{\set{S}}\subseteq \widetilde{\set{V}}$ be such a subset, and define  the 
three subsets 
 \begin{align*}
  \set{S}^{\inOp}       &= \setdef{v\in\set{V}}{v^{\inOp}\in\widetilde{\set{S}} \text{ and } v^{\outOp}\notin\widetilde{\set{S}}}, \\
  \set{S}^{\outOp}      &= \setdef{v\in\set{V}}{v^{\inOp}\notin\widetilde{\set{S}} \text{ and } v^{\outOp}\in\widetilde{\set{S}}}, \text{ and}\\
  \set{S}^{\inOp\outOp} &= \setdef{v\in\set{V}}{v^{\inOp}\in\widetilde{\set{S}} \text{ and } v^{\outOp} \in \widetilde{\set{S}}}
 \end{align*}
of $\set{V}$. Due to~\eqref{eq:specialProp} we find that the left hand side of~\eqref{eq:PSgoal} equals $x(\set{S}^{\outOp})+\gamma$ with $\gamma=1$ if $s^{\inOp}\in\widetilde{\set{S}}$, $t^{\outOp}\not\in\widetilde{\set{S}}$, and $\gamma=0$ otherwise. As the right hand side  of~\eqref{eq:PSgoal} is bounded from below by $x(\set{S}^{\inOp})$, it suffices to ensure $x(\set{S}^{\outOp})+\gamma\le x(\set{S}^{\inOp})$, or, equivalently,
\begin{equation}\label{eq:PSgoal:a}
	x(\set{S}^{\outOp}\cup\set{S}^{\inOp\outOp})+\gamma\le x(\set{S}^{\inOp}\cup\set{S}^{\inOp\outOp})\,.
\end{equation}
Denoting by $\successors{\set{T}}{D}$ the set of all nodes~$w\in\set{V}$ for which there is some $v\in T$ with $(v,w)\in\set{A}$, we find
\begin{equation*}
	\successors{\set{S}^{\outOp}\cup\set{S}^{\inOp\outOp}}{D}\subseteq(\set{S}^{\inOp}\cup\set{S}^{\inOp\outOp})\setminus\{s\}
\end{equation*} 
(due to~\eqref{eq:specialProp} and since~$s$ is a source node). Thus, \eqref{eq:PSgoal:a} follows if 
 \begin{equation}\label{eq:TsuccT}
 	x(\set{T})\le x(\successors{\set{T}}{D})
 \end{equation}
 holds for $T=\set{S}^{\outOp}\cup\set{S}^{\inOp\outOp}$. Indeed, \eqref{eq:TsuccT} obviously is valid for~$\polytope{P}^{s,t}(D)$, unless $t\in T$. Since due to~$x_s=x_t=1$ the difference between the right hand side and the left hand side of~\eqref{eq:PSgoal} remains unchanged when removing~$t^{\outOp}$ from~$\widetilde{\set{S}}$, we thus have established the following linear description. 

\begin{theorem}[Vande Vate~\cite{VandeVate89}]
 For every acyclic digraph $D = (\set{V},\set{A})$ with a source node~$s$ and some other node~$t\ne s$, the following system  provides a linear description of the $s$-$t$-path set polytope~$\polytope{P}^{s,t}(D)$:
 \begin{align*}
	x_s & =  1\\
	x_t & =  1 \\
  x(\set{T}) -x(\successors{\set{T}}{D}) &\leq 0&\quad\forall \set{T}\subseteq \set{V}\ssmin\{t\} \\
  x_v &\geq 0 & \quad\forall v\in\set{V}
 \end{align*}
\end{theorem}
 \section{Polytopes of Small Cliques}
\label{sec: small clique polytopes}

The third example of  polytopes
for which one can easily derive  linear descriptions by means of the lifting method are the polytopes
\begin{equation*}
	\polytope{P}^{\leq 2}(G)=\conv\setDef{\vec{x}[\set{C}]\in\{0,1\}^{\set{V}}}{C\subseteq\set{V}\text{ clique of size }\le 2}
\end{equation*}
associated with  (undirected) graphs $G=(\set{V},\set{E})$. Thus $\polytope{P}^{\leq 2}(G)$ is the convex hull of all characteristic vectors of subsets of nodes that are either empty, singletons, or pairs that form edges. Clearly, $\polytope{P}^{\leq 2}(G)$ equals the  polytope associated with all stable sets of size at most two in the complement of~$G$. In fact, a complete description of these polytopes defined in terms of stable sets has been given by Janssen and Kilakos~\cite{JanssenKilakos99}. 
We will also be concerned with the face
\begin{equation*}
	\polytope{P}^{2}(G)=\conv\setDef{\vec{x}[\{v,w\}]\in\{0,1\}^{\set{V}}}{\{v,w\}\in\set{E}}
\end{equation*}
of~$\polytope{P}^{\leq 2}(G)$ whose vertices are the characteristic vectors of cliques of size exactly two (the \emph{edge-polytope} of the graph~$G$). Before we start, let us briefly consider the dimensions of the introduced polytopes.

\begin{rem}\mbox{}\label{obs: dimension of P^2(G) and P^<=2(G)}
  For every graph $G=(\set{V},\set{E})$ we  have $\dim(\polytope{P}^{\le 2}(G))=|\set{V}|$, thus $\polytope{P}^{\le 2}(G)$ is full-dimensional. The dimension of~$\polytope{P}^2(G)$ (whose affine hull does not contain~$\zeroVec{}$) is one less than the rank of the node-edge incidence matrix of~$G$ (whose columns are the vertices of~$\polytope{P}^2(G)$), where this rank is  $|\set{V}|-\beta(G)$ (see, e.g., \cite{BrualdiRyser91}) with~$\beta(G)$ denoting the number of bipartite connected components of~$G$. Thus we have $\dim(\polytope{P}^2(G))=|\set{V}|-\beta(G)-1$.
\end{rem}

In order to describe the extension of~$\polytope{P}^{\leq 2}(G)$ that we are going to use for a given graph $G=(\set{V},\set{E})$, let us  define a digraph $D = (\set{W},\set{A})$ with a node set $\set{W}$ that again contains two clone nodes~$v^1$ and~$v^2$ of each $v\in\set{V}$, as well as two additional nodes~$s$ and~$t$. We denote $\set{U}^1=\setDef{v^1}{v\in\set{U}}$ and $\set{U}^2=\setDef{v^2}{v\in\set{U}}$ for all $\set{U}\subseteq\set{V}$. 
The arc set $\set{A}$ of $D$ is defined as the set containing
\begin{compactitem}
 \item the arc $(t,s)$,
 \item all arcs pointing from $s$ to $\set{V}^{1}\cup \set{V}^{2}$,
 \item all arcs pointing from from $\set{V}^{1}\cup \set{V}^{2}$ to $t$, and
 \item for any edge $\{v,w\}\in \set{E}$ both the arc $(v^{1},w^{2})$ and the arc $(w^{1},v^{2})$.
\end{compactitem}
Figure~\ref{fig: extension graph for P leq 2 (G)} shows an example for a graph $G$ and its associated digraph $D$.
\begin{figure}
 \psfrag{s}{\scriptsize$s$}              \psfrag{t}{\scriptsize$t$}
 \psfrag{a}{\scriptsize$a$}
 \psfrag{b}{\scriptsize$b$}
 \psfrag{c}{\scriptsize$c$}
 \psfrag{d}{\scriptsize$d$}
 \psfrag{e}{\scriptsize$e$}
 \psfrag{a1}{\scriptsize$a^{(1)}$}              \psfrag{a2}{\scriptsize$a^{(2)}$}
 \psfrag{b1}{\scriptsize$b^{(1)}$}              \psfrag{b2}{\scriptsize$b^{(2)}$}
 \psfrag{c1}{\scriptsize$c^{(1)}$}              \psfrag{c2}{\scriptsize$c^{(2)}$}
 \psfrag{d1}{\scriptsize$d^{(1)}$}              \psfrag{d2}{\scriptsize$d^{(2)}$}
 \psfrag{e1}{\scriptsize$e^{(1)}$}              \psfrag{e2}{\scriptsize$e^{(2)}$}
 \psfrag{V}{\scriptsize$V^{(1)}$}        \psfrag{Vbar}{\scriptsize$V^{(2)}$}
 \includegraphics[width=\textwidth]{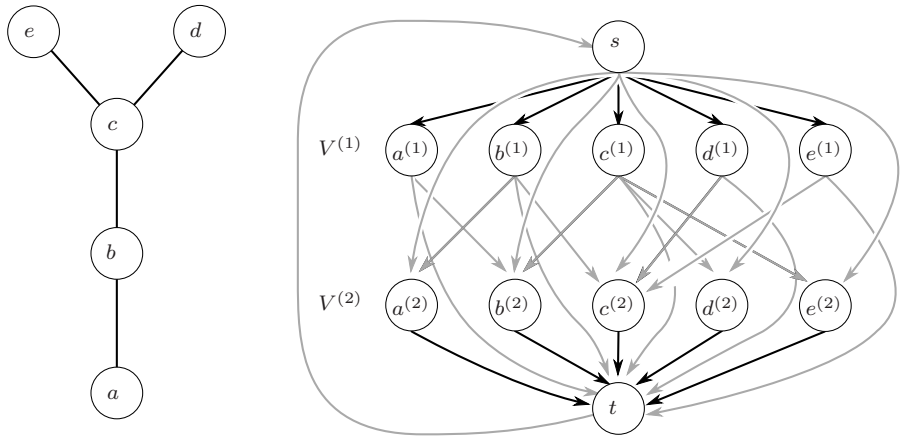}
 \caption{A graph~$G$ and its associated digraph $D$ used in the construction of the extended formulation for~$\polytope{P}^{\leq 2}$. 
 \label{fig: extension graph for P leq 2 (G)}}
\end{figure}

\def\Qstar{\polytope{Q}_{\circOp}(D,\vec{u}^{\star},\vec{\ell}^{\star})}
\def\Qx{\polytope{Q}_{\circOp}(D,\vec{u}^{\vec{x}},\vec{\ell}^{\vec{x}})}
The extension we use is the circulation polytope~$\Qstar$ on~$D$ defined via the following capacities:
\[
 \begin{array}{r@{\;}ll}
  \ell_a^{\star} &= 0       &\text{ for all } a\in\set{A} \\
  u_a^{\star}    &= 1       &\text{ for }     a=(t,s) \\
  u_a^{\star}    &= +\infty &\text{ for all } a\in\set{A}\ssmin\{(t,s)\}
 \end{array}
\]
Again, the vertices of~$\Qstar$ are the characteristic vectors of the directed cycles in~$D$, all of which contain~$(t,s)$. One easily finds that these cycles correspond to the cliques of size at most two in~$G$ (where the empty set is induced by the cycle $\{(s,t),(t,s)\}$, and each clique of size one or two is induced by two cycles). In particular, we have $\sigma(\Qstar)=\polytope{P}^{\leq 2}(G)$ with 
$\sigma: \R^\set{A}\rightarrow\R^\set{V}$ defined via $\sigma(\vec{y})_v=y_{(s,v^{1})}+y_{(v^{2},t)}$.

In order to define a suitable lifting, observe that for every clique~$C\subseteq\set{V}$ of size one or two the most natural choice of a preimage of $\vec{x}[C]$ under the projection~$\sigma$ seems to be the 
average of the two vertices of~$\Qstar$ projected to $\vec{x}[C]$ by~$\sigma$. Therefore, we 
define the lifting $\lambda:\R_+^{\set{V}}\rightarrow\R_+^{\set{A}}$ (with $\set{R}=\R_+^{\set{V}}$ in this case) as follows: For $x\in\R_+^{\set{V}}$ let $\ell^{\vec{x}},u^{\vec{x}}\in\R_+^{\set{A}}$ be lower and upper capacities vectors being equal to~$\ell^{\star}$ and~$u^{\star}$, respectively, in all components except for
\begin{equation*}
	\ell^{\vec{x}}_{(s,v^1)} = \ell^{\vec{x}}_{(v^2,t)} =  
	u^{\vec{x}}_{(s,v^1)} = u^{\vec{x}}_{(v^2,t)} =  
	\frac{x_v}{2}
\end{equation*}
for all $v\in\set{V}$, and choose~$\lambda(\vec{x})$ arbitrarily in~$\Qx$ if the latter set of circulations is non-empty, and (again, just for formal reasons) let $\lambda(\vec{x})\in\R^{\set{A}}$ be an arbitrary point with $\lambda(\vec{x})_{(s,v^1)}=\lambda(\vec{x})_{(v^2,t)}=x_v/2$ for all $v\in\set{V}$ otherwise.

Clearly, $\sigma(\lambda(\vec{x}))=\vec{x}$ holds for all $\vec{x}\in\R_+^{\set{V}}$. Therefore, we only have to find a system of inequalities that is valid for~$\polytope{P}^{\leq 2}(G)$ and section enforcing. As we have $\Qx\subseteq\Qstar$ for all $\vec{x}\in\R_+^{\set{V}}$, the latter condition just means that $\Qx\ne\varnothing$ holds for every $\vec{x}\in\R_+^{\set{V}}$ satisfying that system.
And, by Hoffman's Circulation Theorem~\ref{theo: hoffman's circulation theorem}, for $\vec{x}\in\R_+^{\set{V}}$ we know that $\Qx\ne\varnothing$  is equivalent to
\begin{equation}\label{eq:stab:HoffmanCond}
	\ell^{\vec{x}}(\instar{\set{S}}{D}) \leq u^{\vec{x}}(\outstar{\set{S}}{D}) 
\end{equation}
for all $\set{S}\subset \set{W}$. In fact, if $s\in\set{S}$, then~\eqref{eq:stab:HoffmanCond} is satisfied without any further assumptions on~$\vec{x}\in\R_+^{\set{V}}$ (as then the right-hand side of~\eqref{eq:stab:HoffmanCond} is $+\infty$ if $\set{V}^2\not\subseteq\set{S}$, and the left-hand side of~\eqref{eq:stab:HoffmanCond} is zero, otherwise). Similarly, \eqref{eq:stab:HoffmanCond}  is also satisfied if $t\not\in\set{S}$ holds. Therefore, we only have to ensure by the system to be found that~\eqref{eq:stab:HoffmanCond} holds for all $\set{S}\subseteq\set{W}$ with
\begin{equation}\label{eq:S:cond:1}
	s\not\in\set{S}\quad\text{and}\quad t\in\set{S}\,.
\end{equation}
Among these sets~$\set{S}$, we furthermore only need to consider those with
\begin{equation}\label{eq:S:cond:2}
	\outstar{\set{S}}{D}\cap(\set{V}^1:\set{V}^2)=\varnothing
\end{equation}
(as otherwise the right-hand side of~\eqref{eq:stab:HoffmanCond} again is~$+\infty$). For an arbitrary subset  $\set{S}\subseteq\set{W}$ satisfying~\eqref{eq:S:cond:1} and~\eqref{eq:S:cond:2}, let us partition the original node set~$\set{V}$ into $V=\set{V}_1\uplus\set{V}_2\uplus\set{V}_3\uplus\set{V}_4$ such that we have
\begin{align*}
 \set{V}^1_1\cap\set{S}=\varnothing    &\text{ and } \set{V}^2_1\cap\set{S}=\varnothing &
 \set{V}^1_2\subseteq\set{S} &\text{ and } \set{V}^2_2\cap\set{S}=\varnothing \\
 \set{V}^1_3\cap\set{S}=\varnothing &\text{ and } \set{V}^2_3\subseteq\set{S} &
 \set{V}^1_4\subseteq\set{S} &\text{ and } \set{V}^2_4\subseteq\set{S}.
\end{align*}
Thus we find that the left-hand side of~\eqref{eq:stab:HoffmanCond} evaluates to $\tfrac12 x(\set{V}_1)+x(\set{V}_2)+\tfrac12 x(\set{V}_4)$ and the right-hand side equals one. Hence, we need to find a system of valid inequalities for~$\polytope{P}^{\leq 2}(G)$ ensuring
\begin{equation}
	x(\set{V}_1)+2 x(\set{V}_2)+x(\set{V}_4)\le 2
\end{equation}
for all $\set{S}\subseteq\set{W}$ satisfying~\eqref{eq:S:cond:1} and~\eqref{eq:S:cond:2}. Indeed, \eqref{eq:S:cond:2} implies that~$\set{V}_2$ is a stable set in~$G$ and $\set{V}_1\cup\set{V}_4$ is a subset of $\nonNb{\set{V}_2}$, where, for any subset $\set{T}\subseteq\set{V}$ we denote by $\nonNb{\set{T}}$ the set of all nodes in~$\set{V}\setminus\set{T}$ that are not adjacent to any node from~$\set{T}$. Since the system
\begin{equation}\label{eq:stabTheSystem}
	2x(\set{T})+x(\nonNb{\set{T}})\le 2
	\quad
	\text{for all stable sets }\set{T}\subseteq\set{V}\text{ in }G
\end{equation}
is valid for~$\polytope{P}^{\leq 2}(G)$ (and due to the nonnegativity of~$\vec{x}$), \eqref{eq:stabTheSystem} thus is a system as searched for. 

\begin{theorem}[see also Janssen and Kilakos~\cite{JanssenKilakos99}]
 For every graph $G = (\set{V},\set{E})$, the following set of inequalities provides a complete  linear description for $\polytope{P}^{\leq 2}(G)$:
 \begin{align}
	2x(\set{T})+x(\nonNb{\set{T}}) & \le  2 &
	\forall\set{T}\subseteq\set{V}\text{ stable in }G
	\label{ineq: fDI linear description for P<=2 a}
\\
  x_v &\geq 0 & \forall v\in \set{V} \label{ineq: fDI linear description for P<=2 b}
 \end{align}
	If one restricts~\eqref{ineq: fDI linear description for P<=2 a} to those stable sets~$\set{T}$ for which the subgraph of~$G$ induced by $\nonNb{\set{T}}$ does not have any bipartite connected component, then the description is irredundant.
\end{theorem}

\begin{proof}\mbox{}
	The fact that~\eqref{ineq: fDI linear description for P<=2 a} and~\eqref{ineq: fDI linear description for P<=2 b} provide a complete linear description of~$\polytope{P}^{\leq 2}(G)$ follows from the arguments given above. 
	
	Clearly, all inequalities~\eqref{ineq: fDI linear description for P<=2 b} define facets of~$\polytope{P}^{\leq 2}(G)$, because the face of the~$n$-dimensional polytope~$\polytope{P}^{\leq 2}(G)$ clearly is isomorphic to the $(n-1)$-dimensional polytope~$\polytope{P}^{\leq 2}(G[V\setminus\{v\}])$ (where $G[\set{W}]$ is the subgraph of~$G$ induced by the node subset~$\set{W}$). For every stable set $\set{T}\subseteq\set{V}$ in~$G$, we find that the face defined by the corresponding inequality from~\eqref{ineq: fDI linear description for P<=2 b} contains for each $v\in\set{T}$ the point $\unitvec{v}$ (the point with all components equal to zero except for a one in component~$v$),  for each $v\in\set{V}\setminus(\set{T}\cup\nonNb{\set{T}})$ a point $\unitvec{v}+\unitvec{w}$ for some $w\in\set{T}$, and the set $\setDef{\unitvec{v}+\unitvec{w}}{v,w\in\nonNb{\set{T}},\{v,w\}\in\set{E}}$. Since the latter set is isomorphic to the vertex set of~$\polytope{P}^{\leq 2}(G[\nonNb{\set{T}}])$, we find from Remark~\ref{obs: dimension of P^2(G) and P^<=2(G)} that the dimension of the face we are considering is 
\begin{equation*}
	|\set{T}|+|\set{V}\setminus(\set{T}\cup\nonNb{\set{T}})|+|\nonNb{\set{T}}|-|\beta|-1=|\set{V}|-1-\beta\,,
\end{equation*}
where~$\beta$ is the number of bipartite connected components of~$G[\nonNb{\set{T}}]$. As clearly none of the inequalities in~\eqref{ineq: fDI linear description for P<=2 a} and~\eqref{ineq: fDI linear description for P<=2 b} is a multiple of another one, this proves the statement on irredundancy.
\end{proof}
Note that the characterization of facet defining inequalities given in~\cite{JanssenKilakos99} seems not to be completely correct (as has been noticed by Matthias Peinhardt). For instance, according to the characterization given there, for the graph consisting of  three components, one being an isolated node~$k$, one being a triangle on the  set~$\set{A}_1$ of three nodes and one being an isolated edge on the two-nodes set~$\set{A}_2$, the inequality $2x(\set{K})+x(\set{A})\le 2$ with $\set{K}=\{k\}$ and $\set{A}=\set{A}_1\cup\set{A}_2$ (thus $\set{A}=\nonNb{\set{K}}$) should be facet defining, which it is, however, not, since the subgraph induced by~$\set{A}$ clearly has one bipartite component.

\medskip
  We close this section by providing also an irredundant linear description of the face~$\polytope{P}^2(G)$ of $\polytope{P}^{\leq 2}(G)$ defined by the equation $\vec{x}(\set{V}) = 2$. 
   The face of~$\polytope{P}^2(G)$ defined by an inequality of type~\eqref{ineq: fDI linear description for P<=2 a} is isomorphic to~$\polytope{P}^2(G')$, where~$G'$ is the  graph obtained from~$G$ by removing all edges   
 inside~$\nb{\set{T}}$ as well as all edges connecting~$\nb{\set{T}}$ 
with~$\nonNb{\set{T}}$ 
(where~$\nb{\set{T}}$ is the set of nodes outside~$\set{T}$ adjacent to any node in~$\set{T}$). Thus, denoting 
 by~$\family{T}$ the set of those stable sets~$\set{T}$ in~$G$ such that the number of bipartite connected components increases by exactly one when removing all edges inside~$\nb{\set{T}}$ as well as all edges connecting~$\nb{\set{T}}$ with~$\nonNb{\set{T}}$, we find (again using Remark~\ref{obs: dimension of P^2(G) and P^<=2(G)})
     that~\eqref{ineq: fDI linear description for P<=2 a} defines a facet of~$\polytope{P}^2(G)$ if and only if $\set{T}\in\family{T}$ holds. Moreover, the inequality in~\eqref{ineq: fDI linear description for P<=2 a} defines an implicit equation for~$\polytope{P}^2(G)$ if and only if~$\set{T}$ is from the set~$\family{B}$ of shores of bipartite connected components of~$G$ (where the two \emph{shores} of a bipartite connected component are meant to be the two stable sets into which its node set can be partitioned). Finally, the face of~$\polytope{P}^2(G)$ defined by the inequality in~\eqref{ineq: fDI linear description for P<=2 b} is isomorphic to~$\polytope{P}^2(G[\set{V}\setminus\{v\}])$, and the inequality is an implicit equation for~$\polytope{P}^2(G)$ if and only if~$v$ is an isolated  node in~$G$ (we denote the set of isolated nodes by~$\set{I}$). Denoting  by $\widetilde{\set{V}}$ the subset of all nodes~$v$ for which the number of bipartite components does not increase when removing~$v$ from~$G$, we thus find that $x_v\ge 0$ defines a facet of~$\polytope{P}^2(G)$ if and only if $v\in\widetilde{\set{V}}$ holds. Subtracting, for cosmetic reasons, the equation $x(\set{V})=2$ from the inequalities~\eqref{ineq: fDI linear description for P<=2 a}, we thus have established the following.
 
\begin{theorem}
 For any graph $G = (\set{V},\set{E})$, the following set of inequalities provides a complete non-redundant linear description for $\polytope{P}^{2}(G)$:
 \begin{align*}
  x(\set{V}) &= 2 \\
  x(\set{T})     - x(\nb{\set{T}}{G}) &= 0 \quad\forall\set{T}\in\family{B} \\
  x(\set{T})     - x(\nb{\set{T}}{G}) &\leq 0\quad\forall\set{T}\in\family{T} \\
  x_v                    &= 0\quad \forall v\in\set{I} \\
  x_v                    &\geq 0\quad \forall v\in\widetilde{\set{V}}
 \end{align*}
\end{theorem}

 \section{Orbisacks}
\label{sec:orbi}

Let us denote by $\orbverts{p}$ the set of all  0/1-matrices $\vec{x}\in\{0,1\}^{p\times 2}$ whose first column is lexicographically not smaller than the second one, i.e., for 
\begin{equation*}
	\crit{\vec{x}}=\min(\setDef{i\in\{1,\dots,p\}}{x_{i,1}=1,x_{i,2}=0}\cup\{p+1\})
\end{equation*}
we have
	$x_{i,1}=x_{i,2}$ for all $1\le i<\crit{\vec{x}}$  (the number~$\crit{\vec{x}}$ is the \emph{critical row} of~$x$).
We call the polytope $\orb{p}=\conv{\orbverts{p}}$ an \emph{orbisack}, because it is both an orbitope (see, e.g., \cite{KaibelPfetsch08,Loos11}) and a Knapsack polytope. 

In this section, we will first first identify~$\orb{p}$ as a projection of a polytope~$\polytope{Q}_p^{x,y}$, and then we will identify~$\polytope{Q}_p^{x,y}$ itself as a projection of another polytope~$\polytope{Q}_p^{\tilde{x},y,z}$. For the latter polytope  it will be trivial to find a linear description, hence yielding an extended formulation for~$\polytope{Q}_p^{x,y}$, from which we will derive a linear description of~$\polytope{Q}_p^{x,y}$ by the lifting method. Applying the lifting method once more to the extended formulation of~$\orb{p}$ given by the latter description 
of~$\polytope{Q}_p^{x,y}$ will finally lead us to a linear description of~$\orb{p}$. 
So much for the plan, let's get it done.

In order to define the first extension~$\polytope{Q}_p^{x,y}$,  we append to each vertex $\vec{x}$ of the orbisack~$\orb{p}$ some $0/1$-vector storing information about the position of the critical row of $\vec{x}$. More precisely, we define for each vertex $\vec{x}$ of the orbisack $\orbisack{p,2}$ the vector $\vec{y}(\vec{x})\in\{0,1\}^{[p]}$ via
\[
 \vec{y}(\vec{x}) =  \left\{\begin{array}{rl} 
                       \unitvec{\crit{\vec{x}}}, &\text{ if } \crit{\vec{x}} < p+1 \\ 
                       \zerovec{}, &\text{ if } \crit{\vec{x}} = p+1 
                      \end{array}\right.
\]
(where $\unitvec{i}$ is the point with all components equal to zero except for a one at component~$i$).
Thus
\[
 \polytope{Q}_p^{x,y} = \conv \setdef{(\vec{x},\vec{y}(\vec{x}))\in \R^{[p]\times [2]}\times\R^{[p]}}{\vec{x} \text{ vertex of } \orbisack{p,2}}
\]
clearly provides an extension of~$\orb{p}$ via 
the coordinate projection
 \[
  \sigma:\R^{[p]\times[2]}\times\R^{[p]}\to\R^{[p]\times[2]},\quad (\vec{x},\vec{y}) \mapsto \vec{x}\,.
 \]

For the construction of the extension~$\polytope{Q}_p^{\tilde{x},y,z}$ of~$\polytope{Q}_p^{x,y}$ announced above, we furthermore define for every vertex $\vec{x}$ of $\orbisack{p,2}$ the points
$\tilde{\vec{x}}(\vec{x})\in \{0,1\}^{[p]\times[2]}$ with
\begin{equation*}
\tilde{\vec{x}}(\vec{x}) = 
\begin{cases}
    (x_{i,1},x_{i,2}), & \text{ if } i>\crit{\vec{x}} \\
    (0,0),             & \text{ otherwise }	
\end{cases}
\end{equation*}
and $\vec{z}(\vec{x})\in\{0,1\}^p$ with 
\begin{equation*}
z_i = 
\begin{cases}
    x_{i,1}=x_{i,2}, & \text{ if } i<\crit{\vec{x}} \\
    0,       & \text{ otherwise }	
\end{cases}
\end{equation*}
for all $i\in[p]$.
Thus, $\tilde{\vec{x}}(\vec{x})$ and $\vec{z}(\vec{x})$ store the entries of~$\vec{x}$ below and above the critical row  of~$\vec{x}$, respectively.
It is easy to see that
\[
 \polytope{Q}_p^{\tilde{x},y,z} = \conv \setdef{(\tilde{\vec{x}}(\vec{x}),\vec{y}(\vec{x}),\vec{z}(\vec{x}))\in \R^{[p]\times [2]}\times\R^{[p]}\times\R^{[p]}}{\vec{x} \text{ vertex of } \orbisack{p,2}}
\]
provides an extension of~$\polytope{Q}_p^{x,y}$ via the  projection $\tilde{\sigma}:\R^{[p]\times[2]}\times\R^p\times\R^p$ defined by $(\tilde{\vec{x}},\vec{y},\vec{z})\mapsto (\vec{x},\vec{y})$ with:
\begin{eqnarray}
 x_{i,1} &=& \tilde{x}_{i,1} + y_i + z_i \label{eq: xyz -> xi1} \\
 x_{i,2} &=& \tilde{x}_{i,2} + z_i       \label{eq: xyz -> xi2}
\end{eqnarray}
It turns out that a linear description of~$\polytope{Q}_p^{\tilde{x},y,z}$ is easy to obtain.

\begin{proposition}\label{prop: linear description of P tilde(x)yz}
 The polytope $\polytope{Q}_p^{\tilde{x},y,z}$ is  described by the following set of inequalities
  \begin{align}
  \tilde{x}_{i,1} - \sum_{k=1}^{i-1} y_k &\leq 0 \qquad\forall i\in[p]            \label{ineq: xyz-description 1} \\
  \tilde{x}_{i,2} - \sum_{k=1}^{i-1} y_k &\leq 0 \qquad\forall i\in[p]             \label{ineq: xyz-description 2} \\
  \sum_{k=1}^{i} y_k + z_i &\leq 1               \qquad\forall i\in [p]                       \label{ineq: xyz-description 3} \\
  \tilde{x}_{i,j},y_i,z_i &\geq 0                \qquad\forall i\in [p] \text{ and } j\in [2] \label{ineq: xyz-description 4} 
   \end{align}
\end{proposition}

\begin{proof}
	It is easy to check that the integral points satisfying the system \eqref{ineq: xyz-description 1},\dots,\eqref{ineq: xyz-description 4} are exactly the points whose convex hull is~$\polytope{Q}^{\tilde{x},y,z}$ by definition (note that the system implies $\tilde{x}_{1,1}=\tilde{x}_{1,2}=0$). Since the coefficient matrix of that system is totally unimodular (as it basically is an interval matrix on the~$\vec{y}$-part and the identity matrix on the remaining part) this proves the claim.
\end{proof}

The derivation of a linear description of~$\polytope{Q}_p^{x,y}$ from the extended formulation \eqref{ineq: xyz-description 1},\dots,\eqref{ineq: xyz-description 4} now can be done almost automatically. In order to define a suitable lifting function $\tilde{\lambda}:\tilde{\set{R}}\to\R^{[p]\times[2]}\times\R^p\times\R^p$ with 
\begin{equation*}
	\tilde{\set{R}}=\R^{[p]\times[2]}\times\R_+^p\,,
\end{equation*}
 we first deduce from~\eqref{eq: xyz -> xi1} and~\eqref{eq: xyz -> xi2} that  $\tilde{\sigma}(\tilde{\lambda}(\vec{x},\vec{y}))=(\vec{x},\vec{y})$ holds if and only if we have $\tilde{\lambda}(\vec{x},\vec{y})=(\tilde{\vec{x}},\vec{y},\vec{z})$ with
\begin{eqnarray}
 \tilde{x}_{i,1} &=& x_{i,1} - y_i - z_i \label{eq: xi1 -> xyz} \\
 \tilde{x}_{i,2} &=& x_{i,2} - z_i       \label{eq: xi2 -> xyz}
\end{eqnarray}
for all~$i\in[p]$. Therefore, the only freedom we have in the definition of the lifting is the choice of~$\vec{z}$. Plugging in~\eqref{eq: xyz -> xi1} and~\eqref{eq: xyz -> xi2} (and exploiting the definition of~$\tilde{\set{R}}$), the system~\eqref{ineq: xyz-description 1},\dots,\eqref{ineq: xyz-description 4} (to be satisfied by $(\tilde{\vec{x}},\vec{y},\vec{z})=\tilde{\lambda}(\vec{x},\vec{y})$) turns into
\begin{equation*}
	    \max\{x_{i,1} - \sum_{k=1}^{i} y_k\ ,\ x_{i,2} - \sum_{k=1}^{i-1} y_k\ ,\ 0\}  
	\le z_i
	\le \min\{1-\sum_{k=1}^i y_k\ ,\ x_{i,1} - y_i\ ,\ x_{i,2}\}
\end{equation*}
for all $i\in[p]$. For each~$i\in[p]$, such a~$z_i$ exists if and only if the nine inequalities stating that each of the three expressions taken the maximum over shall not exceed any of the three expressions taking the minimum over are satisfied. Thus, the system made up from  these~$9p$ inequalities is section enforcing. Furthermore, it is clear that this system must be feasible for~$\polytope{Q}_p^{x,y}$ because every point in~$\polytope{Q}_p^{x,y}$ has a preimage in~$\polytope{Q}_p^{\tilde{x},y,z}$  (due to 
$\tilde{\sigma}(\polytope{Q}_p^{\tilde{x},y,z})=\polytope{Q}_p^{x,y}$). Hence, that system together with $\vec{y}\ge\zeroVec{}$ provides a linear description of~$\polytope{Q}_p^{x,y}$. Clearing some redundancies we find the following.

\begin{proposition}\label{prop: linear description of P xy}
 The polytope $\polytope{Q}_p^{x,y}$ is  described by the following system of inequalities (each one occurring for all $i\in[p]$):

  \begin{align}
  x_{i,1}    &\leq 1                                               \label{ineq: xy-description 1} \\
  x_{i,2}    &\geq 0                                                \label{ineq: xy-description 2} \\
  y_i        &\geq 0                                               \label{ineq: xy-description 3} \\
  y_i        &\geq x_{i,1}-x_{i,2}-\sum_{k=1}^{i-1}y_k              \label{ineq: xy-description 4} \\
  y_i        &\leq x_{i,1}                                          \label{ineq: xy-description 5} \\
  y_i        &\leq 1-x_{i,2}                                        \label{ineq: xy-description 6} \\
  y_i        &\leq x_{i,1}-x_{i,2}+\sum_{k=1}^{i-1}y_k              \label{ineq: xy-description 7} \\
  y_i        &\leq 1-\sum_{k=1}^{i-1}y_k                            \label{ineq: xy-description 8} 
   \end{align}
\end{proposition}

From the extended formulation of~$\orb{p}$ provided by the system in Prop.~\ref{prop: linear description of P xy} (via the orthogonal projection~$\sigma$ to the $x$-coordinates) we now finally derive a linear description of~$\orb{p}$ by the lifting method. In order to construct a suitable lifting $\lambda:\set{R}\to\R^{[p]\times[2]}\times\R^p$ with 
\begin{equation*}
	\set{R}=[0,1]^{[p]\times[2]}
\end{equation*}
let us define  (inductively), for each~$\vec{x}\in\set{R}$, the lifting $\lambda(\vec{x})=(\vec{x},\vec{y})$ via
\begin{equation}\label{eq:defymin}
	y_i=\min\{x_{i,1}\ ,\ 1-x_{i,2}\ ,\ x_{i,1}-x_{i,2}+\sum_{k=1}^{i-1}y_k\ ,\ 1-\sum_{k=1}^{i-1}y_k\}
\end{equation}
for each $i\in[p]$ (note that this implies $y_1=x_{1,1}-x_{1,2}$). The idea here is that with this choice of~$\vec{y}$ we only have to find a system of inequalities for~$\vec{x}\in\set{R}$ that enforces~\eqref{ineq: xy-description 3} and~\eqref{ineq: xy-description 4} for all $i\in[p]$ and that is valid for~$\orb{p}$. 

In order to find such a system, suppose~$(\vec{x},\vec{y})$ with~$\vec{x}\in\set{R}$ and~$\vec{y}$ defined as described above does not satisfy all inequalities~\eqref{ineq: xy-description 3} and~\eqref{ineq: xy-description 4}. Let~$i\st$ be the minimal~$i$ for which any of these inequalities is violated. Due to $\vec{x}\in\set{R}$ and the minimality of~$i\st$ we find that~$y_{i\st}$ can neither be equal to~$x_{i\st,1}$ nor to~$1-x_{i\st,2}$. If~$y_{i\st}$ was equal to $1-\sum_{k=1}^{i\st-1}y_k$ then~\eqref{ineq: xy-description 4} was satisfied due to $\vec{x}\in\set{R}$, and~\eqref{ineq: xy-description 3} could not be violated because of $y_{i\st-1}\le 1-\sum_{k=1}^{i\st-2}y_{k}$ in case of~$i\st>1$, and because of $1\ge 0$ in case of $i\st=1$. Thus we have 
\begin{equation*}
	y_{i\st}=x_{i\st,1}-x_{i\st,2}+\sum_{k=1}^{i\st-1}y_k\,,
\end{equation*}
which due to the minimality of~$i\st$ implies that~\eqref{ineq: xy-description 4} is satisfied, hence
\begin{equation}\label{eq:vio ieq}
	x_{i\st,1}-x_{i\st,2}+\sum_{k=1}^{i\st-1}y_k<0
\end{equation}
must hold. 

The strategy now is to expand the left-hand-side of~\eqref{eq:vio ieq} via~\eqref{eq:defymin} into some linear expression in~$\vec{x}$ and to show that all the linear expressions that could arise this way evaluate to nonnegative values for all vertices of~$\orb{p}$, thus constructing a system of valid inequalities for~$\orb{p}$ that prevent us from~\eqref{eq:vio ieq}. Towards this end let us first observe that also for no $i< i\st$ we have $y_i=1-\sum_{k=1}^{i-1}y_k$ (because this would imply $y_{i\st}=1-\sum_{k=1}^{i\st-1}y_k$ due to $0\le y_{i'}\le1-\sum_{k=1}^{i'-1}y_k=0$  for all $i<i'\le i\st$). Let us define a vector $\vec{\tau}(\vec{x})\in\{0,1,2,3\}^p$ with component
$\vec{\tau}(\vec{x})_i$ equal to~$1$ if $y_i=x_{i,1}$, else if $y_i=1-x_{i,2}$ equal to~$2$,  else if $y_i=x_{i,1}-x_{i,2}+\sum_{k=1}^{i-1}y_k$ equal to~$3$, and otherwise equal to~$0$. We call a vector $\vec{\tau}\in\{0,1,2,3\}^p$ \emph{feasible} if $\tau_1=3$ holds and if there is some $i\in[p]$ such that $\tau_{i'}\ne 0$ for all $1\le i'<i$, $\tau_{i}=3$, and $\tau_{i'}=0$ for all $i<i'\le p$. Thus, $\vec{\tau}(\vec{x})$ is feasible. To every feasible vector $\vec{\tau}\in\{0,1,2,3\}^p$ with $i\st=\max\setDef{i}{\tau_i\ne 0}$ we associate two other vectors $\vec{\alpha}=\vec{\alpha}(\vec{\tau})\in\N^p$ and $\vec{a}=\vec{a}(\vec{\tau})\in\R^{[p]\times[2]}$ via
\begin{equation*}
	\alpha_i=
	\begin{cases}
		0 & \text{if }i>i\st \\
		1 & \text{if }i\in\{i\st,i\st-1\} \\
		\alpha_{i+1} & \text{if }i<i\st-1\text{ and }\tau_{i+1}\ne 3\\
		2\alpha_{i+1} & \text{if }i<i\st-1\text{ and }\tau_{i+1}= 3
	\end{cases}
\end{equation*}
and
\begin{equation*}
	(a_{i,1},a_{i,2})=
	\begin{cases}
		(0,0) &                \text{if }\tau_i=0 \\
		(\alpha_i,0) &         \text{if }\tau_i=1 \\
		(0,-\alpha_i) &        \text{if }\tau_i=2 \\
		(\alpha_i,-\alpha_i) & \text{if }\tau_i=3 
	\end{cases}
\end{equation*}
for all $i\in[p]$ as well as a number
\begin{equation*}
	\beta(\vec{\tau})=\sum_{i\,:\,\tau_i=2}\alpha_i\,.
\end{equation*}
With these definitions, we can write the left-hand-side of~\eqref{eq:vio ieq} as
\begin{equation*}
	x_{i\st,1}-x_{i\st,2}+\sum_{k=1}^{i\st-1}y_k=\scalProd{\vec{a}(\vec{\tau}(\vec{x}))}{\vec{x}}+\beta(\vec{\tau}(\vec{x}))\,.
\end{equation*}
Calling, for every feasible $\vec{\tau}\in\{0,1,2,3\}^p$, the inequality 
\begin{equation*}
	\scalProd{-\vec{a}(\vec{\tau})}{\vec{x}}\le \beta(\vec{\tau})
\end{equation*}
a \emph{block inequality} (called \emph{valued} block inequalities in \cite{Loos11}), it thus remains to show that all block inequalities are valid for~$\orb{p}$. But this is easy to see, since a vertex~$\vec{x}$ of~$\orb{p}$ clearly maximizes $\scalProd{-\vec{a}(\vec{\tau})}{\vec{x}}$ among all vertices with prescribed critical row~$i_c\in[p+1]$ if and only if it satisfies
\begin{equation}\label{eq:vertInBlockFace}
	(x_{i,1},x_{i,2})\ 
	\begin{cases}
		=(0,0) & \text{if }i<i_c\text{ and }\tau_i=1 \\
		=(1,1) & \text{if }i<i_c\text{ and }\tau_i=2 \\
		=(0,1) & \text{if }i>i_c\text{ and }\tau_i=3 \\
		\in\{(0,0),(0,1) & \text{if }i>i_c\text{ and }\tau_i=1 \\
		\in\{(0,1),(1,0) & \text{if }i>i_c\text{ and }\tau_i=2 
	\end{cases}\,,
\end{equation}
and in this case, with $\gamma=1$ in case of $\tau_{i_c}\in\{1,3\}$, and $\gamma=0$ otherwise, we have (setting $\alpha_{p+1}=0$)
\begin{equation*}
	\scalProd{-\vec{a}(\vec{\tau})}{\vec{x}}=
	\sum_{i:i\ne i_c,\tau_i=2}\alpha_i
	\ -\gamma\alpha_{i_c}
	+\sum_{i:i>i_c,\tau_i=3}\alpha_i
	\le\beta(\vec{\tau})\,,
\end{equation*}
where the latter inequality follows from $\sum_{i:i>i_c,\tau_i=3}\alpha_i\le\alpha_{i_c}$ (and equality holds unless $i_c=\max\setDef{i}{\tau_i\ne 0}$). Thus, we have established the main part of the following theorem.

\begin{theorem}
	The block inequalities together with the bounds $\zeroVec{}\le\vec{x}\le\oneVec{}$ provide a complete linear description of the orbisack~$\orb{p}$. The only redundant inequalities in this description are $x_{1,1}\ge 0$ and $x_{1,2}\le 1$. 
\end{theorem}

\begin{proof}
	It only remains to prove the statement about redundancy. Let us denote by $\set{F}(i,j,0)$ and $\set{F}(i,j,1)$ the  faces of~$\orb{p}$  defined by $x_{i,j}\ge 0$ and $x_{i,j}\le 1$, respectively, by $\set{L}(i,j,0)$ and $\set{L}(i,j,1)$ the linear subspaces parallel to them (and of the same dimension), and  by  $\set{X}(i,j,0)$ and by $\set{X}(i,j,1)$ the vertex sets of those faces. Clearly, we have $\set{X}(1,1,0)\subsetneq\set{X}(1,2,0)$ and $\set{X}(1,2,1)\subsetneq\set{X}(1,1,1)$, thus both $x_{1,1}\ge 0$ and $x_{1,2}\le 1$ do not define facets of~$\orb{p}$. 
	Every other face $\set{F}(i\st,j\st,\varrho)$ with $\varrho\in\{0,1\}$, however, is a facet of~$\orb{p}$, which one can see as follows. By forming  differences of appropriately chosen pairs from $\set{X}(i\st,j\st,\varrho)$ we find $\unitvec{(i,j)}\in\set{L}(i\st,j\st,\varrho)$ for all $i>1$, $(i,j)\ne (i\st,j\st)$.  In case of $i\st>1$, we similarly find $\unitvec{(1,1)},\unitvec{(1,1)}+\unitvec{(1,2)}\in\set{L}(i\st,j\st,\varrho)$ establishing $\dim(\set{F}(i\st,j\st,\varrho))\ge 2p-1$, and in case of $(i\st,j\st,\varrho)=(1,1,1)$ 
	or $(i\st,j\st,\varrho)=(1,2,0)$ we find $\unitvec{(1,2)}\in\set{L}(i\st,j\st,\varrho)$ or $\unitvec{(1,1)}\in\set{L}(i\st,j\st,\varrho)$, respectively, showing $\dim(\set{F}(i\st,j\st,\varrho))\ge 2p-1$ also for $i\st=1$. 
	
	Hence, denoting by $\set{X}(\vec{\tau})$ the vertex set of the face defined by the block inequality induced by the feasible vector $\vec{\tau}\in\{0,1,2,3\}^p$, we only have to show that $\set{X}(\vec{\tau})$ is neither contained in any $\set{X}(i,j,\varrho)$ nor in any $\set{X}(\vec{\tau}')$ for a feasible vector~$\vec{\tau}'\in\{0,1,2,3\}^p$ different from~$\vec{\tau}$. For a feasible vector $\vec{\tau}\in\{0,1,2,3\}^p$ the set $\set{X}(\vec{\tau})$ consists of those vertices~$\vec{x}$ of~$\orb{p}$ with $i_c\ne i_{\max}$ satisfying~\eqref{eq:vertInBlockFace} for all $i\in[p]$, where $i_c$ is the critical row of~$\vec{x}$ and $i_{\max}=\max\setDef{i}{\tau_i\ne 0}$. Using this characterization, it is easy to construct, for every $(i,j)\in[p]\times [2]$ and $\varrho\in\{0,1\}$  a point~$\vec{x}$ in~$\set{X}(\vec{\tau})$ with $x_{i,j}=\varrho$. Thus, no face defined by a block inequality is contained in any face defined by a trivial inequality.
	
	Finally, let $\vec{\tau},\vec{\tau}'\in\{0,1,2,3\}^p$ be two arbitrary feasible vectors with $\set{X}(\vec{\tau})\subseteq\set{X}(\vec{\tau}')$ and suppose that $\vec{\tau}\ne\vec{\tau}'$ holds. With $i_{\max}=\max\setDef{i}{\tau_i\ne 0}$ and $i'_{\max}=\max\setDef{i}{\tau'_i\ne 0}$ we find $i_{\max}= i'_{\max}$, because otherwise there were vertices in $\set{X}(\vec{\tau})$ with critical row~$i'_{\max}$, thus not contained in $\set{X}(\vec{\tau}')$. In particular, for every $i\in[p]$ with $\tau'_i=0$ we have~$\tau_i=0=\tau'_i$.
	Furthermore,  observe that for every $i\in[p]$ with $\tau'_i\in\{1,2\}$ we must have $\tau_i=\tau'_i$ as well, because otherwise we can easily construct a vertex $\vec{x}\in\set{X}(\vec{\tau})\setminus\set{X}(\vec{\tau}')$ (with critical row~$i_c=p+1$).
	Finally, for every $i\in[p]$ with $\tau'_i=3$ we have $\tau_i=3$ as well, which follows since we have $\tau_{1}=3=\tau'_{1}$ by the definition of feasibility, and since, for $i>1$, we could easily construct some vertex (with critical row~$1$, note $i_{\max}=i'_{\max}\ge i>1$) in $\set{X}(\vec{\tau})\setminus\set{X}(\vec{\tau}')$ in case of $\tau_i\ne 3$.
 	\end{proof}

\section{Conclusions}

The examples worked out in this paper demonstrate  some cases in which it is  convenient to use the lifting method in order to find a linear description of a polytope from an appropriate extended formulation. Other examples where this technique has been used successfully include packing and partitioning orbitopes~\cite{FaenzaKaibel09}. We believe that the technique should be useful in many more situations, as it provides means to exploit knowledge about the vertices of the polytope to describe (e.g., when searching for a suitable lifting function) which seems to be difficult to exploit when working with the projection cone.

\paragraph*{Acknowledgements} We thank Matthias Peinhardt for many valuable discussions on the topic of the paper. 


\begin{thebibliography}{10}

\bibitem{BrualdiRyser91}
Richard~A. Brualdi and Herbert~John Ryser.
\newblock {\em Combinatorial Matrix Theory}.
\newblock Cambridge University Press, 1991.

\bibitem{CCZ10}
Michele Conforti, G{{\'e}}rard Cornu{{\'e}}jols, and Giacomo Zambelli.
\newblock Extended formulations in combinatorial optimization.
\newblock {\em 4OR}, 8(1):1--48, 2010.

\bibitem{FaenzaKaibel09}
Yuri Faenza and Volker Kaibel.
\newblock Extended formulations for packing and partitioning orbitopes.
\newblock {\em Mathematics of Operations Research}, 34(3):686--697, 2009.

\bibitem{Hoffman60}
Alan~J. Hoffman.
\newblock Some recent applications of the theory of linear inequations to
  extremal combinatorial analysis.
\newblock In Richard~E. Bellman and Marshall Hall, editors, {\em Proceedings of
  Symposia in Applied Mathematics, American Mathematical Society, Providence},
  volume~10, pages 113--127, 1960.

\bibitem{JanssenKilakos99}
Jeannette Janssen and Kyriakos Kilakos.
\newblock Bounded stable sets: Polytopes and colorings.
\newblock {\em Siam Journal of Discrete Mathematics}, 12(2):262--275, 1999.

\bibitem{Kai11}
Volker Kaibel.
\newblock Extended formulations in combinatorial optimization.
\newblock {\em Optima}, 85:2--7, 2011.

\bibitem{KaibelPfetsch08}
Volker Kaibel and Marc~E. Pfetsch.
\newblock {Packing and partitioning orbitopes}.
\newblock {\em Mathematical Programming}, 114(1):1--36, 2008.

\bibitem{Loos11}
Andreas Loos.
\newblock {\em Describing Orbitopes by Linear Inequalities and Projection Based
  Tools}.
\newblock PhD thesis, University of Magdeburg, 2011.

\bibitem{Schrijver04}
Alexander Schrijver.
\newblock {\em Combinatorial Optimization (Polyhedra and Efficiency)}, volume
  A-C.
\newblock Springer, 2004.

\bibitem{VandeVate89}
John~H. {Vande Vate}.
\newblock {The path set polytope of an acyclic, directed graph with an
  application to machine sequencing}.
\newblock {\em Networks}, 19(5):607--614, 1989.

\bibitem{VW10}
Francois Vanderbeck and Laurence~A. Wolsey.
\newblock Reformulation and decomposition of integer programs.
\newblock In Michael J{{\"u}}nger, Thomas Liebling, Denis Naddef, George
  Nemhauser, William Pulleyblank, Gerhard Reinelt, Giovanni Rinaldi, and
  Laurence Wolsey, editors, {\em 50 years of integer programming 1958--2008},
  pages 431--502. Springer, 2010.

\end{thebibliography}

\end{document}